\documentclass[12pt]{article}
\usepackage{amsfonts}
\usepackage{graphicx}
\usepackage{amsmath}
\usepackage{amssymb}
\usepackage{a4wide}

\newcounter{main}

\def\st{{}^\star}  
\def\starr{{}^\star \mathbb{R}}   
\def\starn{{}^\star \mathbb{N}}   

\newtheorem{theorem}{Theorem}[section]
\newtheorem{proposition}[theorem]{Proposition}
\newtheorem{lemma}[theorem]{Lemma}

\newtheorem{definition}{Definition}[section]

\newcommand{\blanksquare}{\,\,\,$\sqcup\!\!\!\!\sqcap$}
\newenvironment{proof}{{\flushleft {\bf Proof: }}}{\blanksquare}

\newcounter{example}
{{\stepcounter{example}}{\flushleft {\bf Example \arabic{example}:}}}%
{\par}

\title{{\bf Sharkovskii order for non-wandering points}}
\author{ Maria Carvalho and Fernando Moreira\thanks{Partially supportedby FCT through
CMUP }}
\date{ }

\begin{document}

\maketitle

\begin{abstract}
For a map $f:I \rightarrow I$, a point $x \in I$ is periodic with period $p \in \mathbb{N}$ if $f^p(x)=x$ and $f^j(x)\not=x$ for all $0<j<p$. When $f$ is continuous and $I$ is an interval, a theorem due to Sharkovskii (\cite{BC}) states that there is an order in $\mathbb{N}$, say $\lhd$, such that, if $f$ has a periodic point of period $p$ and $p \lhd q$, then $f$ also has a periodic point of period $q$. In this work, we will see how an extension of this order $\lhd$ to an ultrapower of the integer numbers yields a Sharkovskii-type result for non-wandering points of $f$.
\end{abstract}

\emph{keywords:} Sharkovskii order; transfer principles.\\
\normalsize

\begin{section}{Introduction}

Let $f:[a,b] \rightarrow \mathbb{R}$ be a continuous map. A point $x_0 \in [a,b]$ is \emph{non-wandering} if, for each neighborhood $\mathcal{V}$ of $x_0$, there is a positive integer $N$ such that $f^N(\mathcal{V})\, \cap\, \mathcal{V}\neq \emptyset$. If $f^k(\mathcal{V})\, \cap \, \mathcal{V}=\emptyset$ for all $k \in \{1,2,\cdots,N-1\}$, we say that $N$ is a \emph{first return} of $\mathcal{V}$ to itself. This notion is a weak form of recurrence and gathers recurrent points (the ones that are accumulated by their orbits) and the periodic ones. The aim of this work is to generalize Sharkovskii's theorem to non-wandering points, where periodic points are replaced by neighborhoods of the non-wandering point, and periods by return times. The main difficulty of such a formulation lies on the control of the speed of the return and its nearness to the starting point, parameters that, in the case of a periodic orbit, are not only elementary to express but completely determined by 
 the period.

A straight extension of Sharkovskii's result should state that, given sequences $(R_n)_{n\in\mathbb{N}}$ and $(S_n)_{n\in\mathbb{N}}$ of positive integers related by the ultrapower extension of the order $\lhd$, if $f:[a,b] \rightarrow \mathbb{R}$ has a non-wandering point $x_0$ with a fundamental system of neighborhoods whose first returns happen at times $(R_n)_{n\in\mathbb{N}}$, then $f$ has a non-wandering point $x_1$ with a fundamental system of neighborhoods returning at times $(S_n)_{n\in\mathbb{N}}$. This is true, but uninteresting, if the sequences $(R_n)_{n\in\mathbb{N}}$ and $(S_n)_{n\in\mathbb{N}}$ are eventually constant (equal to $c$ and $d$, respectively). In fact, if each neighborhood $\mathcal{V}_{n}$ of a fundamental system at $x_0$ has a first return by the power $f^c$, then there is $y_{n} \in \mathcal{V}_{n}$ (so the sequence $(y_n)_{n \in \mathbb{N}}$ converges to $x_{0}$) such that $f^{c}(y_{n}) \in \mathcal{V}_{n}$ (the sequence $\left(f^{c}(y_{n})\right)_{n \in \mathbb{N}}$ also converges to $x_0$), and, therefore, as $f$ is continuous, $x_0$ is periodic with period $c$; so, as $c \lhd d$, Sharkovskii's theorem informs that $f$ has a periodic point with period $d$, to whom we may easily find a fundamental system of neighborhoods first returning by $f^d$.

However, for more general sequences of returns, our argument demands a strict mastery of the size of the neighborhoods with respect to the amount of time a return needs to occur.

\begin{theorem}\label{maintheorem}
Let $f:[a,b] \rightarrow \mathbb{R}$ be a continuous function and $x_0$ a non-wandering point of $f$. Consider two sequences, $(R_n)_{n\in\mathbb{N}}$ and $(S_n)_{n\in\mathbb{N}}$, of positive integers related by the nonstandard extension of the order $\lhd$. Given $\epsilon > 0$, there exists a sequence of positive real numbers $(\delta\,(\epsilon,S_n))_{n\in\mathbb{N}}$ converging to zero such that, if the fundamental system of neighborhoods $(\mathcal{V}_n)_{n\in\mathbb{N}}\, = \left(\left]x_0-\delta\,(\epsilon,\,S_n), x_0 + \delta\,(\epsilon,\,S_n)\right[\right)_{n\in\mathbb{N}}$ has first returns at times $(R_n)_{n\in\mathbb{N}}$, then there exists a subsequence $(S_{n_k})_{k \in \mathbb{N}}$ and a non-wandering point $x_1$ of $f$ which has a fundamental system of neighborhoods returning at times $(S_{n_k})_{k \in \mathbb{N}}$.
\end{theorem}

In this work, we apply techniques coming from nonstandard analysis in such a way that, infinite returns of a set to itself may be described by a periodic point of a suitable dynamics acting on the set of hyperreals, $\starr$, to where the statement of Sharkovskii's theorem may be conveyed. The proof goes as follows: starting with a non-wandering point $x_0$, we create, through small local perturbations of $f$, a sequence of periodic points of a countable family of dynamics transferable to $\starr$; here, the corresponding dynamical system is continuous and has a periodic point with period equal to the hyperinteger represented by the sequence $(R_n)_{n\in\mathbb{N}}$; then, we apply the nonstandard version of Sharkovskii's theorem, getting another periodic point with period given by the hyperinteger associated to $(S_n)_{n\in\mathbb{N}}$; finally, this information is projected on $f$ and $\mathbb{R}$, thus emerging the requested non-wandering point $x_1$. This strategy has its
  cost: in general, we can not ensure that $x_1 \neq x_0$ nor that $(S_{n_k})_{k \in \mathbb{N}}$ are first returns.

\medskip

\end{section}

\begin{section}{Preliminaries from nonstandard analysis}

In this section we present the nonstandard tools used in the proof of Theorem \ref{maintheorem}. We will not need all the strength of a general nonstandard framework; therefore we will follow the approach to basic nonstandard analysis from \cite{G}.

\begin{subsection}{The space of hyperreals}

The hyperreals, denoted by $\starr$, is the structure obtained by an ultrapower  construction of the real numbers $\mathbb R$. The building process begins considering the space of real sequences, $\mathbb{R}^\mathbb{N}$, and the equivalence relation given by
$$(\alpha_n)_{n\in\mathbb{N}} \equiv  (\beta_n)_{n\in\mathbb{N}} \Leftrightarrow \{n\in\mathbb{N}:\, \alpha_n=\beta_n\} \text{ is big},
$$

\noindent where \emph{big} means that the set belongs to a fixed  ultrafilter\footnote{Given a nonempty subset $\mathcal{S}$, an ultrafilter of $\mathcal{S}$ is a collection $\mathfrak{F}$ of subsets of $\mathcal{S}$ such that $\emptyset \notin \mathfrak{F}$; $A,\, B \in \mathfrak{F} \Rightarrow A\,\cap \,B \in \mathfrak{F}$; $A \in \mathfrak{F} \text{ and } A \subseteq B \subseteq \mathcal{S} \Rightarrow B \in \mathfrak{F}$; $\forall A \subseteq \mathcal{S} \text{ either } A \in \mathfrak{F} \text{ or } \mathcal{S}-A \in \mathfrak{F} \text{, but not both by the two first properties}.$} of $\mathbb{N}$ that contains all the co-finite sets. In the sequel, the quotient space $\mathbb{R}^\mathbb{N}/\equiv $ will be denoted by $\starr$. Within it, each real number $\alpha$ is identified with the equivalence class of the constant sequence equal to $\alpha$, say $[(\alpha)_{n \in \mathbb{N}}]$. Besides, it may be endowed with a sum operation
$$[(\alpha_n)_{n \in \mathbb{N}}]\oplus[(\beta_n)_{n \in \mathbb{N}}]=[(\alpha_n + \beta_n)_{n \in \mathbb{N}}],$$

\noindent a product
$$[(\alpha_n)_{n \in \mathbb{N}}]\otimes[(\beta_n)_{n \in \mathbb{N}}]=[(\alpha_n \times \beta_n)_{n \in \mathbb{N}}]$$

\noindent and an order defined by
$$[(\alpha_n)_{n \in \mathbb{N}}] \prec [(\beta_n)_{n \in \mathbb{N}}] \Leftrightarrow \{n \in \mathbb{N}: \alpha_n < \beta_n\} \text{ is big},$$
$$[(\alpha_n)_{n \in \mathbb{N}}] \preceq [(\beta_n)_{n \in \mathbb{N}}] \Leftrightarrow \{n \in \mathbb{N}: \alpha_n \leq \beta_n\} \text{ is big}.$$

\noindent With this structure, $(\starr,\oplus, \otimes, \prec, \preceq)$ is an ordered field with infinitesimals\footnote{A hyperreal $[(r_n)_{n \in \mathbb{N}}]$ is an infinitesimal if and only if $\lim_{n \rightarrow \infty}\, r_n=0$.} and unlimited numbers.\footnote{A hyperreal $[(r_n)_{n \in \mathbb{N}}]$ is unlimited if and only if $\lim_{n \rightarrow \infty}\, r_n=\infty$; it is the inverse of an infinitesimal.}
\end{subsection}

\medskip

\begin{subsection}{Extension of sets and functions}

Let $A$ be a subset of $\mathbb{R}$. By $\st A$, we denote the subset of $\starr$ defined by the condition
$$[(\alpha_n)_{n\in\mathbb{N}}] \in \st A \Leftrightarrow \{n\in\mathbb{N}:\alpha_n\in A\} \text{ is big}.$$

Given a map $f:A\,\subseteq\,\mathbb{R}\,\rightarrow\,\mathbb{R}$, its extension to $\st A$ is the function $\st f: \, \st A\,\rightarrow\,\starr$ that assigns to each $[(\alpha_n)_{n\in\mathbb{N}}]\in \st A$ the hyperreal $[(f(\alpha_n))_{n\in\mathbb{N}}]$. There are hyperfunctions which are not extensions of maps of a real variable (so called \emph{not internal}), as the map that assigns to each bounded hyperreal $\theta$ its \emph{shadow} \footnote{The shadow of a bounded hyperreal $\theta$ is the unique real number $t$ such that $\theta \oplus [(-t)_{n \in \mathbb{N}}]$ is an infinitesimal.}.

\medskip

With these two notions, we may extend the concepts of sequence and dynamical system.

\begin{enumerate}
\item
A \emph{hypersequence} is a map $S:\starn\,\rightarrow\,\starr$.

Therefore, on the sequel, the term $S_N$ is also defined for hyperintegers $N$.
\item
A \emph{discrete dynamical system} in $\starr$ is a map
$F:\starn\times\starr\,\rightarrow\,\starr$ such that, for each $X\in\starr$ and all pair $N,M\in\starn$, we have
$F(N\,\oplus \,M,X)=F(N,F(M,X))$.

Accordingly, an element $\alpha$ in $\starr$ is \emph{periodic} by $F$ with period $P\in\starn$ if and only if $F(P,\alpha)=\alpha$ and $F(Q,\alpha)\neq \alpha$ for every $Q \in \starn$ such that $[(1)_{n \in \mathbb{N}}] \preceq Q \prec P$.
\end{enumerate}

\end{subsection}

\begin{subsection}{Nonstandard closeness}

Two hyperreals $\theta$ and $\vartheta$ are infinitely close to each other (abbreviated into $\theta \approx \vartheta$) if $\theta-\vartheta$ is an infinitesimal. This condition defines in $\starr$ an equivalence relation, say $\propto$, whose classes are called \emph{halos}. This notion is useful to reformulate real assertions in nonstandard notation, or by suggesting on how to extend them to the nonstandard realm. For instance:

\begin{enumerate}
\item
A map $f:\mathbb{R}\,\rightarrow\,\mathbb{R}$ is \emph{continuous} at $c\in\mathbb{R}$ if and only if for any $x\in\mathbb{R}$ such that $[(x)_{n \in \mathbb{N}}]\approx [(c)_{n \in \mathbb{N}}]$ we have $[(f(x))_{n \in \mathbb{N}}]\approx [(f(c))_{n \in \mathbb{N}}]$.

Similarly, a map $F$ on the hyperreals is continuous at a point $C$ if and only if $F\bigl(\textit{halo}(C)\bigr) \subseteq \textit{halo}\bigl(F(C)\bigr).$
\item
A sequence of real numbers $(\alpha_n)_{n\in\mathbb{N}}$ \emph{converges} to $L\in\mathbb{R}$ is and only if $[(\alpha_n)_{n \in \mathbb{N}}] \approx [(L)_{n \in \mathbb{N}}]$.

Analogous definition in $\starr.$
\item
$L\in\mathbb{R}$ is an \emph{accumulation point} of the real sequence $(\alpha)_{n\in\mathbb{N}}$ if and only if there is a non-bounded\footnote{A hyperreal $[(t_n)_{n \in \mathbb{N}}]$ is bounded if and only if there is a real $M$ such that the set $\{n \in \mathbb{N}: t_n<M \}$ is big.} element in $\starn$ such that $\alpha_N \approx [(L)_{n \in \mathbb{N}}].$
\item
Given a map $f:A\,\subseteq\,\mathbb{R}\,\rightarrow\,\mathbb{R}$, a point $x_0 \in A$ is \emph{non-wandering} if there are $z \in \starr$ and $N \in \starn$, say $N=[(m_1,m_2,\cdots,m_k,\cdots)]$, such that $z \approx [(x_0)_{n \in \mathbb{N}}]$ and $f^N(z) \approx [(x_0)_{k \in \mathbb{N}}]$, where $f^N(z)=[(f^{m_1}(z), f^{m_2}(z), \cdots, f^{m_k}(z), \cdots)]$.

If we are allowed to choose $z=[(x_0)_{n \in \mathbb{N}}]$, then $x_0$ is said to be \emph{recurrent}.

If $z=[(x_0)_{n \in \mathbb{N}}]$ and $N=[(p)_{n \in \mathbb{N}}]$ for a positive integer $p$, then $x_0$ is \emph{periodic} with period $p$.
\end{enumerate}
\end{subsection}

\begin{subsection}{Transfer principles}
The construction of $\starr$ endows it with logical principles that allow the exchange of notions and valid properties between $\mathbb{R}$ and $\starr$. More precisely, we have:

\begin{enumerate}
\item \emph{\textbf{The $\star$-transfer.}}

\noindent The field $\mathbb{R}$ has a complete structure $\mathcal{R}=\ll \mathbb{R},\,Rel(\mathbb{R}),\,Fun(\mathbb{R})\gg$, where $Rel(\mathbb{R})$ represents all the finite relations on $\mathbb{R}$ and $Fun(\mathbb{R})$ all the real maps with real variable. Accordingly, to $\starr$ we assign the structure $\st \mathcal{R}=\ll\starr,\{\st \mathcal{P}:\mathcal{P}\in Rel(\mathbb{R})\},\{\st f:f\in Fun(\mathbb{R})\}\gg$. This way, $\st \mathcal{R}$ is the natural extension of all the relations and maps in $\mathbb{R}$. However, this mere extension does not ensure completeness of $\st \mathcal{R}$ since, for instance, the map $f:\starr\,\rightarrow\,\starr$ given by $f(x)=\epsilon^x$, where $\epsilon$ is an infinitesimal, is not an extension of any element of $Fun(\mathbb{R})$. Anyway, this is a mechanism that transforms a formula $\varphi$ of the language $\mathcal{L}_\mathbb{R}$ with relational symbols $\mathcal{P}$ and functionals $\mathcal{F}$ into another formula of a la
 nguage $\mathcal{L}_{\starr}$ from the structure $\st \mathcal{R}$, replacing $\mathcal{P}$ by $\st \mathcal{P}$ and $\mathcal{F}$ by $\st \mathcal{F}$. For example, the proposition
$$[\forall x\in\mathbb{R}\,\,\exists n\in\mathbb{N}\,\,:\,x<n]$$
yields
$$[\forall X\in\starr\,\,\exists N\in\starn\,\,:\,X\,\prec\, N].$$

\medskip

\item \emph{\textbf{The universal transfer and its dual.}}

\noindent The principle that rules this operator asserts that \emph{if a property is valid for all reals, then it is valid for all hyperreals}. That is,
$(\forall x\in\mathbb{R})\,\varphi \rightarrow (\forall X\in\starr)\,\st \varphi$.

\medskip

Dually, \emph{if there exists an hyperreal verifying some property, than there is a real one that also satisfies it}. That is, $(\exists X\in\starr)\,\st \varphi \Rightarrow (\exists x\in\mathbb{R})\,\varphi.$

\medskip

The main example we have in mind is Sharkovskii's theorem. Its full extension to the hyperreals depends on the lift of Sharkovskii's ordering
$$
3 \lhd  5 \lhd  7 \lhd  \ldots  \lhd  (2n+1)  \lhd  (2n+3)  \lhd  \ldots $$
$$
\ldots \lhd 2 \times 3  \lhd  2 \times 5  \lhd  2 \times 7  \lhd  \ldots  \lhd 2 \times  (2n+1)  \lhd  2 \times ( 2n+3 ) \lhd  \ldots $$
$$
\ldots \lhd 2^2 \times 3  \lhd  2^2 \times 5  \lhd  2^2 \times 7  \lhd  \ldots  \lhd 2^2 \times  (2n+1)  \lhd  2^2 \times ( 2n+3 ) \lhd  \ldots $$
$$\vdots$$
$$
\ldots \lhd 2^\ell \times 3  \lhd  2^\ell \times 5  \lhd  2^\ell \times 7  \lhd  \ldots  \lhd 2^\ell \times  (2n+1)  \lhd  2^\ell \times ( 2n+3 ) \lhd  \ldots $$
$$\vdots$$
$$
\ldots \lhd 2^\ell \lhd  2^{\ell-1} \lhd \ldots  \lhd  2^3  \lhd  2^2   \lhd  1$$
to the nonstandard setting. The $\star$-transfer of the positive integers and corresponding positions in this ordering
$$[(3)_{n \in \mathbb{N}}] \,\, \st \lhd [(5)_{n \in \mathbb{N}}] \,\, \st \lhd \cdots \st \lhd [(2)_{n \in \mathbb{N}}]\otimes [(3)_{n \in \mathbb{N}}] \,\, \st \lhd [(2)_{n \in \mathbb{N}}]\otimes [(5)]_{n \in \mathbb{N}} \,\, \st \lhd \cdots \st \lhd \cdots$$
$$\cdots \st \lhd [(2^n)_{n \in \mathbb{N}}] \,\, \st \lhd [(2^{n-1})_{n \in \mathbb{N}}] \,\, \st \lhd \cdots [(2)_{n \in \mathbb{N}}] \,\, \st \lhd [(1)_{n \in \mathbb{N}}]$$
works finely and suggests how to pursue.

\begin{definition}
Two elements $\theta=[(t_n)_{n \in \mathbb{N}}]$ and $\vartheta=[(u_n)_{n \in \mathbb{N}}]$ in $\starn$ verify the relation $\theta \,\, \st \lhd \vartheta$
if and only if the set $\{n \in \mathbb{N}:\,t_n \lhd u_n\}$ is big.
\end{definition}

Moreover, as $\lhd$ is complete in $\mathbb{N}$, the universal transfer guarantees its extension to all other hypernatural numbers of $\starn$. Thus the nonstandard version of Sharkovskii's theorem states that:

\begin{theorem}\label{theoremSark}
Let $F:\starn\times\starr\,\rightarrow\,\starr$ be a dynamical system in $\starr$. Consider two sequences of positive integers, say $[(R_n)_{n \in \mathbb{N}}]$ and $[(S_n)_{n \in \mathbb{N}}]$, verifying $[(R_n)_{n \in \mathbb{N}}] \,\, \st \lhd  [(S_n)_{n \in \mathbb{N}}]$. If $F$ has a periodic point with period $[(R_n)_{n \in \mathbb{N}}]$, then $F$ has a periodic point with period $[(S_n)_{n \in \mathbb{N}}].$
\end{theorem}

\end{enumerate}
\end{subsection}
\end{section}

\begin{section}{The size of the neighborhoods}

In the sequel we will consider the uniform norm $\|\,g\,\|=\max\,\{|g(x)|:a\leq x\leq b\}$ on the space of real continuous maps defined on the interval $[a,b]$. Let $f:[a,b]\,\rightarrow\,\mathbb{R}$ be such a map. As $f$ is uniformly continuous in $[a,b]$, given $\tau>0$ there exists $\eta(\tau)>0$ such that, if $x$ and $y$ belong to $[a,b]$ and $|x-y|<\eta(\tau)$, then $|f(x)-f(y)|<\frac{\tau}{2}$.

\medskip

Consider a  non-wandering point $x_0$ of $f$, a sequence $(S_n)_{n\in\mathbb{N}}$ of positive integers and $\epsilon>0$.

\begin{definition}\label{delta}
For each $n \in \mathbb{N}$, the number $\delta(\epsilon,S_n)$ is the maximum of the set with $S_n + 1$ elements given by
$$\Bigl\{\frac{1}{n},\frac{\epsilon}{2},\frac{\eta(\epsilon)}{2},\frac{\eta(\eta(\epsilon))}{2},\cdots,\frac{\eta(\eta(\eta(\cdots\eta(
\epsilon)\cdots)))}{2}\Bigr\}.$$
\end{definition}

\noindent Notice that, this way, $0<\delta(\epsilon,S_n)\leq \frac{1}{n}.$ Besides,

\medskip

\begin{lemma}\label{Lema1}
Given $n \in \mathbb{N}$, if $\|f-g\|<\delta(\epsilon,S_n)$, then $\|f^k-g^k\|<\epsilon$ for all $k \in \{1,2,\cdots,S_n\}$.
\end{lemma}

\begin{proof} Let us fix $n$ and the corresponding $S_n$. If $k=1$, the assertion is a direct consequence of the fact that $\delta(\epsilon,S_n)<\epsilon$. For $k=2$, as $\|f-g\|<\delta(\epsilon,S_n)$, we know that $|f-g\|<\frac{\epsilon}{2}$ and $\|f-g\|<\eta(\epsilon)$. Therefore, by the definition of $\eta(\epsilon)$, we have $\|f\circ f-f\circ g\|<\frac{\epsilon}{2}$; besides, $\|f\circ g-g\circ g\|\leq\|f-g\|<\frac{\epsilon}{2}$. So
$$\|f^2-g^2\|\leq\|f\circ f-f\circ g\|+\|f\circ g-g\circ g\|<\epsilon.$$
Similarly, from $\|f-g\|<\delta(\epsilon,S_n)$, we deduce that $\|f-g\|<\frac{\epsilon}{2}$, $\|f-g\|<\frac{\eta(\epsilon)}{2}$ and $\|f-g\|<\eta(\eta(\epsilon))$, which together imply, as just checked, that
$$\|f^2-g^2\|\leq\|f\circ f-f\circ g\|+\|f\circ g-g\circ g\|<\frac{\eta(\epsilon)}{2}+\frac{\eta(\epsilon)}{2}=\eta(\epsilon),$$
and so
$$\|f^3-g^3\|\leq\|f\circ f^2-f\circ g^2\|+\|f\circ g^2-g\circ g^2\|<\frac{\epsilon}{2}+\frac{\epsilon}{2}.$$
The argument proceeds inductively.
\end{proof}
\end{section}

\begin{section}{A related dynamical system in $\starr$}

Let $f:[a,b] \rightarrow \mathbb{R}$ be a continuous function and $x_0$ a non-wandering point of $f$. Consider two sequences, $(R_n)_{n\in\mathbb{N}}$ and $(S_n)_{n\in\mathbb{N}}$, of positive integers related by $\st \lhd$ and assign to them, according to Definition \ref{delta}, the sequence $\left(\delta(\epsilon,S_n)\right)_{n \in \mathbb{N}}$. Take the fundamental system of neighborhoods of $x_0$ given by
$$(\mathcal{V}_n)_{n\in\mathbb{N}}\, = \left(]x_0-\delta\,(\epsilon,\,S_n), x_0 + \delta\,(\epsilon,\,S_n)[\right)_{n\in\mathbb{N}},$$
whose first returns happen at times $(R_n)_{n\in\mathbb{N}}$. Since $f^{R_n}(\mathcal{V}_n)\,\cap\,\mathcal{V}_n\neq\emptyset$, we may choose a sequence $(y_n)_{n\in\mathbb{N}}$ of elements of $[a,b]$ such that, for all $n$, $y_n$ and $f^{R_n}(y_n)$ both belong to $\mathcal{V}_n$.


\medskip

\begin{proposition}\label{Prop1}
For each $n$, there is a continuous map $g_n:[a,b]\,\rightarrow\,\mathbb{R}$ such that $g^{R_n}_n(y_n)=y_n$ and $\|f-g_n\|<\delta(\epsilon,S_n)$.
\end{proposition}

\begin{proof}
As $y_n$ and $f^{R_n}(y_n)$ belong to the open interval $\mathcal{V}_n$, we may find a $\zeta>0$ such that the intervals $I_n$ (the one that connects these two points inside $[a,b]$, which may be degenerated if $y_n=f^{R_n}(y_n)$) and $J_n$ (which we obtain from $I_n$ adding to it two short segments, with length $\zeta$, on its extremes) are contained in $\mathcal{V}_n$. Consider a continuous bump-function $\phi_n:[a,b]\,\rightarrow\,\mathbb{R}$ so that the restriction of $\phi_n$ to $I_n$ is constant and equal to $1$, and the value of $\phi_n$ in $[a,b]\backslash J_n$ is zero. Denote by $T_n:[a,b]\,\rightarrow\,\mathbb{R}$ the map
$$T_n(t)=t+\bigl[y_n-f^{R_n}(y_n)\bigr]\times\phi_n(t).$$
The function $T_n$ is the identity in the complement of $J_n$ and translates the elements of $J_n$ of an amount that does not exceeds $|y_n-f^{R_n}(y_n)|$.

Define now the map $g_n:[a,b]\,\rightarrow\,\mathbb{R}$ by $g_n=T_n \circ f$. This is a continuous function and, as $R_n$ is the first return of $\mathcal{V}_n$, we have
$$g^{R_n}_n(y_n)=(T_n \circ f)^{R_n}(y_n)=T_n(f^{R_n}(y_n))=$$
$$=f^{R_n}(y_n)+\bigl[y_n-f^{R_n}(y_n)\bigr]\times 1 = y_{n}$$
and, for any positive integer $\ell$ such that $1 \leq \ell < R_n$,
$$g^{l}_n(y_n)=T_n(f^{l}(y_n))=f^{l}(y_n)\neq y_{n}.$$
Besides, $g_n$ coincides with $f$ in $[a,b]\backslash f^{-1}(\mathcal{V}_n)$ since, if $t\notin f^{-1}(\mathcal{V}_n)$, then $\phi(f(t))=0$ and therefore
$$g_n(t)=T_n(f(t))=f(t)+\bigl[y_n-f^{R_n}(y_n)\bigr]\times \phi(f(t))=f(t).$$
Moreover, if $t\in f^{-1}(\mathcal{V}_n)$, then
$$|g_n(t)-f(t)|=\left|f(t)+\bigl[y_n-f^{R_n}(y_n)\bigr]\times \phi(f(t))-f(t)\right|=$$
$$=\left|\bigl[y_n-f^{R_n}(y_n)\bigr]\times\phi(f(t))\right|\leq \left|y_n-f^{R_n}(y_n)\right|<\delta(\epsilon,S_n).$$
So $\|g_n-f\|<\delta(\epsilon,S_n).$
\end{proof}

\medskip

\begin{definition}
Denote by $\mathcal{G}:\st [a,b]\,\rightarrow\,\starr$ the map that assigns to each equivalence class $[{(t_n)}_{n \in \mathbb{N}}]$ in $\st [a,b]$ the class $\bigl[g_1(t_1),g_2(t_2),\cdots,g_n(t_n),\cdots\bigr].$
\end{definition}

\medskip

It is straightforward to verify that:
\begin{lemma}
$\mathcal{G}$ is well defined, internal and continuous.
\end{lemma}

\begin{proof}
Let $(s_n)_{n\in\mathbb{N}}$ and $(t_n)_{n\in\mathbb{N}}$ be sequences in the same equivalence class of $\st [a,b]$. This means that $s_m=t_m$ for a \emph{big} set of positive integers $m$. Then, for the same set, we have $g_m(s_m)=g_m(t_m)$ and, therefore, the classes $\bigl[g_1(s_1),g_2(s_2),\cdots,g_n(s_n),\cdots\bigr]$ and $\bigl[g_1(t_1),g_2(t_2),\cdots,g_n(t_n),\cdots\bigr]$ also coincide. $\mathcal{G}$ is internal since it is the class of the sequence of real maps $(g_n)_{n\in\mathbb{N}}$. Moreover, by transfer of the continuity of each $g_n$, we get the continuity of $\mathcal{G}$: for all $\alpha$ in $\st [a,b]$, we have
$\mathcal{G}\Bigl(\text{\emph{halo}}(\alpha)\Bigr)\subset \text{\emph{halo}} \Bigl(\mathcal{G}(\alpha)\Bigr)$.
\end{proof}

\medskip

Consider now in $\starr$ the equivalence classes of the sequences $(R_n)_{n\in\mathbb{N}}$, $(S_n)_{n\in\mathbb{N}}$ and $(y_n)_{n\in\mathbb{N}}$, say $P=[R_1,R_2,\cdots,R_n,\cdots]$, $Q=[S_1,S_2,\cdots,S_n,\cdots]$ and $y_0=[y_1,y_2,\cdots,y_n,\cdots]$.

\medskip

\begin{definition}
$\mathbb{G}:\starn_0\times\st [a,b]\,\rightarrow\,\starr$ denotes the dynamical system associated with the map $\mathcal{G}$, given by
$$\mathbb{G}\Bigl([n_1,n_2,\cdots,n_k,\cdots],[t_1,t_2,\cdots,t_k,\cdots]\Bigr)=\bigl[g^{n_1}_1(t_1),g^{n_2}_2(t_2),\cdots,g^{n_k}_k(t_k),\cdots\bigr].$$
\end{definition}

Notice that, if $\st n=[n_1,n_2,\cdots,n_k,\cdots]$, $\st m=[m_1,m_2,\cdots,m_k,\cdots]$ and $\st n\oplus \st  m=\bigl[n_1+m_1,n_2+m_2,\cdots,n_k+m_k,\cdots\bigr]$, then
$$\mathbb{G}\Bigl(\st n \oplus \st m,[t_1,t_2,\cdots,t_k,\cdots]\Bigr)=\bigl[g^{n_1+m_1}_1(t_1),g^{n_2+m_2}_2(t_2),\cdots,g^{n_k+m_k}_k(t_k),\cdots\bigr]$$
and
$$\mathbb{G}\Bigl(\st n,\mathbb{G}(\st m,[t_1,t_2,\cdots,t_k,\cdots])\Bigr)=\mathbb{G}\Bigl(\st n,\bigl[g^{m_1}_1(t_1),g^{m_2}_2(t_2),\cdots,g^{m_k}_k(t_k),\cdots\bigr]\Bigr)=$$
$$=\bigl[g_1^{n_1}(g^{m_1}_1(t_1)),g_2^{n_2}(g^{m_2}_2(t_2)),\cdots,g_k^{n_k}(g^{m_k}_k(t_k)),\cdots\bigr]\Bigr)=$$
$$\bigl[g^{n_1+m_1}_1(t_1),g^{n_2+m_2}_2(t_2),\cdots,g^{n_k+m_k}_k(t_k),\cdots\bigr]=\mathbb{G}\Bigl(\st n \oplus \st m,[t_1,t_2,\cdots,t_k,\cdots]\Bigr).$$

\end{section}

\medskip

\begin{section}{Proof of Theorem~\ref{maintheorem}}

\begin{proof}
By construction, as $(R_n)_{n \in \mathbb{N}}$ is a sequence of first returns, the dynamical system $\mathbb{G}$ has a periodic point with period $P$:
$$\mathbb{G}\Bigl(P,y_0\Bigr)=\mathbb{G}\Bigl([R_1,R_2,\cdots,R_n,\cdots],[y_1,y_2,\cdots,y_n,\cdots]\Bigr)=$$
$$=\bigl[g^{R_1}_1(y_1),g^{R_2}_2(y_2),\cdots,g^{R_n}_n(y_n),\cdots\bigr]=[y_1,y_2,\cdots,y_n,\cdots]=y_0.$$
Thus, by Theorem~\ref{theoremSark}, $\mathbb{G}$ has a periodic point $z_0=[z_1,z_2,\cdots,z_n,\cdots]$ with period $Q$. Take an accumulation point $x_1$ of $(z_n)_{n\in\mathbb{N}}$ in $[a,b]$ (observing that $x_1$ may coincide with $x_0$) and consider a sequence of positive integers ${(n_k)}_{k\in\mathbb{N}}$ verifying the condition
$$\forall k \in \mathbb{N} \, \, \, z_{n_k}\in \,\, ]x_1-\frac{1}{2n_k},x_1+\frac{1}{2n_k}[.$$
By Proposition~\ref{Prop1}, for each $n_k$ and $\epsilon=\frac{1}{2n_k}$, we have $\|g_{n_k}-f\|<\delta(\epsilon,S_{n_k})$, and so, by Lemma~\ref{Lema1}, we get $\|g_{n_k}^{S_{n_k}}-f^{S_{n_k}}\|<\epsilon=\frac{1}{2n_k}$. Consequently, as $z_{n_k}$ is a periodic point with period $S_{n_k}$ by the iteration of the map $g_{n_k}$, the neighborhood $\mathcal{W}_k=\,]x_1-\frac{1}{n_k},x_1+\frac{1}{n_k}[$ of $x_1$ returns to itself by $f^{S_{n_k}}$. Hence $x_1$ is the non-wandering point we were looking for.
\end{proof}
\end{section}

\bigskip

\flushleft
\emph{Maria Carvalho} \ \  (mpcarval@fc.up.pt)\\
\emph{Fernando Moreira} \ \  (fsmoreir@fc.up.pt)\\
CMUP and Departamento de Matem\'atica \\
Rua do Campo Alegre, 687 \\ 4169-007 Porto \\ Portugal\\

\end{document}